\documentclass[12pt]{article}

 \usepackage{amsfonts,amssymb,amsthm,graphicx,amsmath}
\usepackage[colorlinks=true,linkcolor=blue,citecolor=blue]{hyperref}

\def\line#1{\hbox to \hsize{#1\hfill}}
\newtheorem{prop}{Proposition}

\newtheorem{defi}{Definition}

\newtheorem{theo}{Theorem}
\newtheorem{coro}{Corollary}

\newcommand{\M}{{\cal M}}

\date{27 May 2013 }
\def\mail#1#2{{\tt #1}@{\tt #2}}

\title{On minimal Lagrangian surfaces in the product of Riemannian two manifolds}
\author{Nikos Georgiou\footnote{The author is supported by Fapesp (2010/08669-9)}}

\begin{document}

\maketitle

\centerline{\textbf {\large{Abstract}}}

\bigskip

{\it Let $(\Sigma_1,g_1)$ and $(\Sigma_2,g_2)$ be connected, complete and orientable Riemannian two manifolds. Consider the two canonical K\"ahler structures $(G^{\epsilon},J,\Omega^{\epsilon})$ on the product 4-manifold $\Sigma_1\times\Sigma_2$ given by $ G^{\epsilon}=g_1\oplus \epsilon g_2$, $\epsilon=\pm 1$ and $J$ is the canonical product complex structure. Thus for $\epsilon=1$ the K\"ahler metric $G^+$ is Riemannian while for $\epsilon=-1$, $G^-$ is of neutral signature. 

We show that the metric $G^{\epsilon}$ is locally conformally flat iff the Gauss curvatures $\kappa(g_1)$ and $\kappa(g_2)$ are both constants satisfying $\kappa(g_1)=-\epsilon\kappa(g_2)$. 

We also give conditions on the Gauss curvatures for which every $G^{\epsilon}$-minimal Lagrangian surface is the product $\gamma_1\times\gamma_2\subset\Sigma_1\times\Sigma_2$, where $\gamma_1$ and $\gamma_2$ are geodesics of $(\Sigma_1,g_1)$ and $(\Sigma_2,g_2)$, respectively. Finally, we explore the Hamiltonian stability of projected rank one Hamiltonian $G^{\epsilon}$-minimal surfaces.}

\bigskip

\centerline{\small \em 2000 MSC: 53D12,  49Q05
\em }


\section*{Introduction}

A submanifold of a symplectic manifold is said to be \emph{Lagrangian} if it is half the ambient dimension and the symplectic form vanishes on it. A Lagrangian submanifold of a pseudo-Riemannian manifold is said to be \emph{minimal} if it is a critical point of the volume functional associated with pseudo-Riemannian metric. A minimal submanifold is characterized by the vanishing of the trace of its second fundamental form, the \emph{mean curvature}. Recently, interest in minimal Lagrangian submanifolds in pseudo-Riemannian K\"ahler structures has grown amongst geometers \cite{An2} \cite{Ur2}, while minimal Lagrangian submanifolds in Calabi-Yau manifolds are of great interest in theoretical physics because of their close relationship to mirror symmety \cite{SYZ}. 

In addition, the space ${\mathbb L}({\mathbb M}^3)$ of oriented geodesics in a 3-dimensional space form $({\mathbb M}^3,g)$ admits a natural k\"ahler structure where the metric $G$ is of neutral signature, scalar flat and locally conformally flat \cite{AGK},\cite{An4}\cite{GG1} \cite{GK1}. 

The significance of these structures is that the identity component of the isometry group of $G$ is isomorphic with the identity component of the isometry group of $g$. Moreover, Salvai has proved that the neutral K\"ahler metrics on ${\mathbb L}({\mathbb E}^3)$ and ${\mathbb L}({\mathbb H}^3)$ are the unique metrics with this property \cite{salvai0},\cite{salvai1}. 

The neutral K\"ahler structure on ${\mathbb L}({\mathbb M}^3)$ plays an important role in surface theory in $({\mathbb M}^3,g)$. In particular, if $S$ is a smoothly immersed surface in $M$, the set of oriented geodesics normal to $S$ form a Lagrangian surface in ${\mathbb L}({\mathbb M}^3)$. A Lagrangian surface $\Sigma$ in ${\mathbb L}({\mathbb M}^3)$ is $G$-minimal if and only if $\Sigma$ is locally the set of normal oriented geodesics of an equidistant tube along a geodesic in $M$ \cite{An4} \cite{AGR} \cite{Ge}.

Oh in \cite{Oh1} has introduced a natural variational problem, apart from the classical variational problem of minimizing the volume functional in a homology class, consisting of minimizing the volume with respect to Hamiltonian compactly supported variations. An important property of these variations is that they preserve the Lagrangian constraint.  A Lagrangian submanifold in a K\"ahler or a pseudo-K\"ahler manifold is said to be \emph{Hamiltonian minimal submanifold} if it is a critical point of the volume functional with respect to Hamiltonian compactly supported variations. A Hamiltonian minimal submanifold can be characterized by its mean curvature vector being divergence-free. 

For example, in the space ${\mathbb L}({\mathbb E}^3)$ of oriented lines in the Euclidean 3-space, a Hamiltonian minimal surface is the set of oriented lines normal to a surface $S\subset {\mathbb E}^3$ that is a critical point of the functional 
\[
{\cal F}(S)=\displaystyle\int_S\sqrt{H^2-K}dA,
\] 
where $H, K$ denote the mean and the Gauss curvature of $S$, respectively \cite{AGR}.  

The neutral K\"ahler structures on the space of oriented great circles in the three sphere ${\mathbb S}^3$ and the space of  oriented space-like geodesics in the anti De Sitter 3-space $\mbox{Ad}{\mathbb S}^3$ can both be identified with the product structures, ${\mathbb L}({\mathbb S}^3)={\mathbb S}^2\times {\mathbb S}^2$ and ${\mathbb L}^+(\mbox{Ad}{\mathbb S}^3)={\mathbb H}^2\times {\mathbb H}^2$. 

More generally, one is led to consider the K\"ahler structures derived by the product structure of $\Sigma_1\times\Sigma_2$, where $(\Sigma_1,g_1)$ and $(\Sigma_2,g_2)$ are complete, connected, orientable Riemannian 2-manifolds. 

Let $\omega_1$ and $\omega_2$ be the symplectic two forms of $(\Sigma_1,g_1)$ and $(\Sigma_2,g_2)$ respectively, and let $j_1$ and $j_2$ be their complex structures as Riemann surfaces. 

For $\epsilon=1$ or $-1$, consider the product structures of the four-dimensional manifold $\Sigma_1\times\Sigma_2$ endowed with the product metrics $G^{\epsilon}=\pi_1^{\ast} g_1+ \epsilon \pi_1^{\ast} g_2$, the almost complex structure $J=j_1\oplus j_2$ and the symplectic two forms $\Omega^{\epsilon}=\pi_1^{\ast}\omega_1+ \epsilon \pi_2^{\ast}\omega_2$, where $\pi_i$ are the projections of $\Sigma_1\times\Sigma_2$ onto $\Sigma_i$, $i=1,2$. The quadruples $(\Sigma_1\times\Sigma_2, G^{\epsilon},J,\Omega^{\epsilon})$ are easily seen to be 4-dimensional K\"ahler structures.

In this paper we study $G^{\epsilon}$-minimal Lagrangian surfaces in the K\"ahler $4$-manifold $(\Sigma_1\times \Sigma_2,G^{\epsilon},J,\Omega^{\epsilon})$. 
In section \ref{s:construction} we prove:

\vspace{0.1in}

\noindent
{\bf Theorem 1.} \emph{The K\"ahler metric $G^{+}$ is Riemannian while the K\"ahler metric $G^{-}$ is neutral. Moreover, the  K\"ahler metric $G^{\epsilon}$ is conformally flat if and if the Gauss curvatures $\kappa(g_1)$ and $\kappa(g_2)$ are both constants with $\kappa(g_1)=-\epsilon \kappa(g_2)$. }

\vspace{0.1in}

In section \ref{s:rankoneranotwo}, we first define the \emph{projected rank} (see Definition \ref{d:defiprojectedrank}) of a surface in $\Sigma_1\times\Sigma_2$ and we prove that every Lagrangian surface is either of projected rank one or of projected rank two. 

For the projected rank one case, we classify all Hamiltonian $G^{\epsilon}$-minimal surfaces:

\vspace{0.1in}

\noindent
{\bf Theorem 2.} \emph{Every projected rank one Lagrangian surface can be locally parametrised by $\Phi:S\rightarrow \Sigma_1\times\Sigma_2:(s,t)\mapsto (\phi(s),\psi(t))$, where $\phi$ and $\psi$ are regular curves on $\Sigma$ and the induced metric $\Phi^{\ast}G^{\epsilon}$ is flat. $\Phi$ is Hamiltonian $G^{\epsilon}$-minimal if and only if $\phi$ and $\psi$ are Cornu spirals of parameters $\lambda_{\phi}$ and $\lambda_{\psi}$, respectively, such that
\[
\lambda_{\phi}=-\epsilon\lambda_{\psi}.
\]
$\Phi$ is $G^{\epsilon}$-minimal Lagrangian if and only if  both $\phi$ and $\psi$ are geodesics. Furthermore, every projected rank one $G^{\epsilon}$-minimal Lagrangian surface in $\Sigma_1\times\Sigma_2$ is totally geodesic.}

\vspace{0.1in}

In the same section, the following theorem gives the conditions for the non-existence of projected rank two $G^{\epsilon}$-minimal Lagrangian surfaces:

\vspace{0.1in}

\noindent
{\bf Theorem 3.} \emph{Let $(\Sigma_1,g_1)$ and $(\Sigma_2,g_2)$ be Riemannian two manifolds and let $(G^{\epsilon},J,\Omega^{\epsilon})$ be the canonical K\"ahler product structures on $\Sigma_1\times\Sigma_2$. Let $\kappa(g_1),\kappa(g_2)$ be the Gauss curvatures  of $g_1$ and $g_2$ respectively. Assume that either of the following hold:}

\noindent \emph{(i) The metrics $g_1$ and $g_2$ are both generically non-flat and $\epsilon \kappa(g_1)\kappa(g_2)<0$ away from flat points.}

\noindent \emph{(ii) Only one of the metrics $g_1$ and $g_2$ is flat while the other is non-flat generically. }

\emph{Then every $G^{\epsilon}$-minimal Lagrangian surface is of projected rank one.}

\vspace{0.1in}

Here a generic property is one that holds almost everywhere. Note that Theorem \ref{t:notflattt} is no longer true when $(\Sigma_1,g_1)$ and $(\Sigma_2,g_2)$ are both flat, since there exist projected rank two minimal Lagrangian immersions in the complex Euclidean space ${\mathbb C}^2$ endowed with the pseudo-Hermitian product structure \cite{AGR}.  

Minimality is the first order condition for a submanifold to be volume-extremizing in its homology class. Harvey and Lawson \cite{HL1} have proven that minimal Lagrangian submanifolds of a Calabi-Yau manifold is calibrated, which implies by Stokes theorem, that are volume-extremizing. The second order condition for a minimal submanifold to be volume-extremizing was first derived by Simons \cite{Si}. 

Minimal submanifolds that are local extremizers of the volume are called \emph{stable minimal submanifolds}. The stability of a minimal submanifold is determined by the monotonicity of the second variation of the volume functional. If the second variation of the volume functional of a Hamiltonian minimal submanifold is monotone for any Hamiltonian compactly supported variation, it is said to be \emph{Hamiltonian stable}. 
In \cite{Oh1} and \cite{Oh2}, the second variation formula of a Hamiltonian minimal submanifold has been derived in the case of a K\"ahler manifold, while for the pseudo-K\"ahler case it has been derived in \cite{AnGer}.

The following Theorem in the section \ref{s:hamiltstabilitysection}, investigates the Hamiltonian stability of projected rank one Hamiltonian $G^{\epsilon}$-minimal surfaces in $\Sigma_1\times\Sigma_2$:

\vspace{0.1in}

\noindent
{\bf Theorem 4.} \emph{Let $\Phi=(\phi,\psi)$ be of projected rank one Hamiltonian $G^{\epsilon}$-minimal immersion in $(\Sigma_1\times\Sigma_2,G^{\epsilon})$ such that $\kappa(g_1)\leq -2k_{\phi}^2$ and $\kappa(g_2)\leq -2k_{\psi}^2$ along the curves $\phi$ and $\psi$ respectively. Then $\Phi$ is a  local minimizer of the volume in its Hamiltonian isotopy class.}

\vspace{0.2in}

\noindent {\bf Acknowledgements.} The author would like to thank H. Anciaux and B. Guilfoyle for their helpful and valuable suggestions and comments.

\vspace{0.2in}

\section{The Product K\"ahler structure}\label{s:construction}
Consider the Riemannian 2-manifolds $(\Sigma_k, g_k)$ for $k=1,2$ and denote by $j_k$ the rotation by an angle $+\pi/2$ in $T\Sigma_k$. Set $\omega_k(\cdot, \cdot)=g_k(j_k\cdot, \cdot)$ so that the quadruples $(\Sigma_k,g_k,j_k,\omega_k)$ are 2-dimensional K\"ahler manifolds. 

Using the following identification,
\[
X\in T(\Sigma_1\times\Sigma_2)\simeq (X_1,X_2)\in T\Sigma_1\oplus T\Sigma_2,\qquad \mbox{where}\qquad X_k\in T\Sigma_k.
\]
we obtain the natural splitting $T(\Sigma_1\times\Sigma_2)=T\Sigma_1\oplus T\Sigma_2$. Let $(x,y)\in \Sigma_1\times\Sigma_2$ and $X=(X_1,X_2)$ and $Y=(Y_1,Y_2)$ be two tangent vectors in $T_{(x,y)}(\Sigma_1\times\Sigma_2)$. Define the metric $G^{\epsilon}_{(x,y)}$ by:
\[
G^{\epsilon}_{(x,y)}(X,Y)=g_1(X_1,Y_1)(x)+\epsilon g_2(X_2,Y_2)(y),
\]
where $\epsilon=1$ or $-1$. The Levi-Civita connection $\nabla$ with respect to the metric $G^{\epsilon}$ is 
\[
\nabla_X Y=(D^1_{X_1} Y_1,D^2_{X_2} Y_2),
\]
where $X=(X_1,X_2),Y=(Y_1,Y_2)$ are vector fields in $\Sigma_1\times\Sigma_2$ and $D^1,D^2$ denote the Levi-Civita connections with respect to $g_1$ and $g_2$, respectively. 

Consider the endomorphism $J\in End(T\Sigma_1\oplus T\Sigma_2)$ defined by $J=j_1\oplus j_2$, i.e.,
$J(X)=(j_1X_1,j_2X_2)$. Clearly, $J$ is an almost complex structure on $\Sigma_1\times\Sigma_2$. 
\begin{prop}
The almost complex structure $J$ is integrable.
\end{prop}
\begin{proof}
The Nijenhuis tensor $N_J$ is
\[
N_J(X,Y)=[JX,JY]^{\nabla}-J[JX,Y]^{\nabla}-J[JX,Y]^{\nabla}-[X,Y]^{\nabla},
\]
where $X=(X_1,X_2),Y=(Y_1,Y_2)$ are vector fields in $\Sigma_1\times\Sigma_2$ and $[\cdot,\cdot]^{\nabla}$ denotes the Lie bracket with respect to the Levi-Civita connection $\nabla$. Then 
\[
[X,Y]^{\nabla}=([X_1,Y_1]^{D^1},[X_2,Y_2]^{D^2}), 
\]
where $[\cdot,\cdot]^{D^i}$ are the Lie brackets with respect to the Levi-Civita connections $D^i$. Thus,
\begin{eqnarray}
N_J(X,Y)&=&[JX,JY]^{\nabla}-J[JX,Y]^{\nabla}-J[JX,Y]^{\nabla}-[X,Y]^{\nabla}\nonumber \\
&=&(N_{j_1}(X_1,Y_1),N_{j_2}(X_2,Y_2)),\nonumber
\end{eqnarray}
and the Proposition follows.
\end{proof}

Let $\pi_i:\Sigma_1\times \Sigma_2\rightarrow \Sigma_i$ be the $i$-th projection, and define the following two-forms 
\[
\Omega^{\epsilon}=\pi_1^{\ast}\omega_1+\epsilon\pi_2^{\ast}\omega_2.
\]
\begin{theo} The quadruples $(\Sigma_1\times\Sigma_2, G^{\epsilon},J,\Omega^{\epsilon})$ are 4-dimensional K\"ahler structures. 

The  K\"ahler metric $G^{\epsilon}$ is conformally flat if and if the Gauss curvatures $\kappa(g_1)$ and $\kappa(g_2)$ are both constants with $\kappa(g_1)=-\epsilon \kappa(g_2)$. 
\end{theo}
\begin{proof}
We have already seen that the almost complex structure $J$ is integrable. It is obvious that $\Omega^{\epsilon}$ is closed, i.e., $d\Omega^{\epsilon}=0$ and is therefore a symplectic form on $\Sigma_1\times\Sigma_2$. 

Moreover, $J$ is compatible with $\Omega^{\epsilon}$ since for $X=(X_1,X_2)$ and $Y=(Y_1,Y_2)$, we have
\begin{eqnarray}
\Omega^{\epsilon}_{(x,y)}(JX,JY)&=&\Omega^{\epsilon}_{(x,y)}((j_1X_1,j_1X_2),(j_2Y_1,j_2Y_2))\nonumber \\
&=&\omega_1(j_1X_1,j_1Y_1)(x)+\epsilon\omega_2(j_2X_2,j_2Y_2)(y)\nonumber\\
&=&\omega_1(X_1,Y_1)(x)+\epsilon\omega_2(X_2,Y_2)(y)\nonumber\\
&=&\Omega^{\epsilon}_{(x,y)}(X,Y).\nonumber
\end{eqnarray}

We proceed with the proof by considering the cases of $G^+$ and $G^-$.

\vspace{0.1in}

\noindent {\bf The case of $G^+$:} Assume that $(e_1,e_2)$ and $(v_1,v_2)$ are orthonormal frames on $\Sigma_1$ and $\Sigma_2$ respectively, both oriented in such a way $j_1e_{1}=e_2$ and $j_2v_{1}=v_{2}$. Consider the  orthonormal frame $(E_1,E_2,E_3, E_{4})$ of $(\Sigma_1\times\Sigma_2,G^+)$ defined by 
\[
E_1=\frac{1}{\sqrt{3}}(e_1,v_1+v_2),\qquad E_2=JE_1=\frac{1}{\sqrt{3}}(e_2,v_2-v_1)
\]
\[
 E_3=\frac{1}{\sqrt{3}}(e_1-e_2,-v_1),\qquad E_4=JE_3=\frac{1}{\sqrt{3}}(e_1+e_2,-v_2).
\]
If $Ric^{+}$ denotes the Ricci curvature tensor with respect to the metric $G^+$, we have
\[
Ric^+(E_1,E_1)_{(x,y)}=Ric^+(E_2,E_2)_{(x,y)}=\frac{\kappa(g_1)(x)+2\kappa(g_2)(y)}{3},
\]
\[
Ric^+(E_3,E_3)_{(x,y)}=Ric^+(E_4,E_4)_{(x,y)}=\frac{2\kappa(g_1)(x)+\kappa(g_2)(y)}{3},
\]
and therefore the scalar curvatute $R^+$ is:
\begin{equation}\label{e:scpos}
R^+=\sum_{i=1}^4 Ric^{+}(E_i,E_i)=2(\kappa(g_1)(x)+\kappa(g_2)(y)).
\end{equation}
If $G^{\epsilon}$ is conformally flat, it is scalar flat \cite{De} and thus, from (\ref{e:scpos}), the Gauss curvatures $\kappa(g_1), \kappa(g_2)$ are constants with $\kappa(g_1)=-\kappa(g_2)$.

Conversely suppose that
\begin{equation}\label{e:equalconstcurv}
\kappa(g_1)=-\kappa(g_2)=c,
\end{equation}
where $c$ is a real constant. Consider the corresponding coframe ${\cal B}_+=(e_1,e_2,e_3,e_4)$ of the orthonormal frame $(E_1,E_2,E_3, E_{4})$. The Hodge star operator $\ast:\Lambda^2(\Sigma_1\times\Sigma_2)\rightarrow \Lambda^2(\Sigma_1\times\Sigma_2)$ defined by
\[
a\wedge \ast b=G^+(a,b)\mbox{Vol},
\]
splits the bundle of 2-forms $\Lambda^2(\Sigma_1\times\Sigma_2)$ into:
\[
\Lambda^2(\Sigma_1\times\Sigma_2)=\Lambda^2_+(\Sigma_1\times\Sigma_2)\oplus \Lambda^2_-(\Sigma_1\times\Sigma_2),
\]
where $\Lambda^2_+(\Sigma_1\times\Sigma_2),\Lambda^2_-(\Sigma_1\times\Sigma_2)$ are the self-dual and the anti-self-dual 2-form bundles, respectively and $\mbox{Vol}=e_1\wedge e_2\wedge e_3\wedge e_4$ is the volume element.

With respect to this splitting the Riemann curvature operator ${\cal R}:\Lambda^2(\Sigma_1\times\Sigma_2)\rightarrow \Lambda^2(\Sigma_1\times\Sigma_2)$ defined by
\[
{\cal R}(e_i\wedge e_j)e_k\wedge e_l=G(R(E_i,E_j)E_k,E_l),
\]
is decomposed by:
\[
{\cal R}=\begin{pmatrix} W^{+} +\frac{R^+}{12}I & Z \\  Z^{\ast} & W^{-} +\frac{R^+}{12}I 
\end{pmatrix},
\]
where $W^{\pm}:\Lambda^2_{\pm}(\Sigma_1\times\Sigma_2)\rightarrow \Lambda^2_{\pm}(\Sigma_1\times\Sigma_2)$ are the self-dual and the anti-self-dual part of the Weyl tensor $W$ and $Z$ is the traceless Ricci tensor. Note that $W=W^{+}\oplus W^{-}$. An orthonormal basis for $\Lambda^2_{\pm}(\Sigma_1\times\Sigma_2)$ is
\begin{eqnarray}
e^{\pm}_1&=&\frac{1}{\sqrt{2}}(e_1\wedge e_2\pm e_3\wedge e_4),\nonumber \\
e^{\pm}_2&=&\frac{1}{\sqrt{2}}(e_1\wedge e_3\mp e_2\wedge e_4),\nonumber \\
e^{\pm}_3&=&\frac{1}{\sqrt{2}}(e_1\wedge e_4\pm e_2\wedge e_3).\nonumber 
\end{eqnarray}
The fact that $G^+$ is scalar flat, the self-dual part $W^{+}$ vanishes since,
\[
W^{+}=R^+\begin{pmatrix} 1/3 & & \\ & -1/6 & \\ & & -1/6\end{pmatrix}.
\]
Substituting (\ref{e:equalconstcurv}) into (\ref{e:scpos}) the scalar curvature $R^+$ vanishes and thus $W^-(e^-_i,e^-_j)={\cal R}(e^-_i)e^-_j$. A brief computation shows that ${\cal R}(e^-_i)e^-_j=0$ for all $i,j$. For example,
\begin{eqnarray}
{\cal R}(e^-_1)e^-_2&=&\frac{1}{2}{\cal R}(e_1\wedge e_2)e_1\wedge e_2+\frac{1}{2}{\cal R}(e_3\wedge e_4)e_3\wedge e_4\nonumber \\
&=&\frac{1}{2}\Big(G^+(R(E_1,E_2)E_1,E_2)+G^+(R(E_3,E_4)E_3,E_4)\Big)\nonumber \\
&=&0.\nonumber 
\end{eqnarray}
Thus, the anti-self-dual part $W^-$ also vanishes. Therefore the Weyl tensor $W=0$, or $G^+$ is locally conformally flat.

\vspace{0.1in}

\noindent {\bf The case of $G^-$:} We now prove that the neutral K\"ahler metric $G^-$ is conformally flat if and only if the Gauss curvatures $\kappa(g_1),\kappa(g_2)$ are both constants with $\kappa(g_1)=\kappa(g_2)$. For this metric, consider the orthonormal frame $(E_1,E_2,E_3,E_4)$ defined by: 
\[
E_{1}=(e_{1},v_{1}+v_{2}),\qquad E_{2}=JE_{1}=(e_{2},v_{2}-v_{1}),
\]
\[
 E_{3}=(e_{1}-e_{2},v_{1}),\qquad E_{4}=JE_{3}=(e_{1}+e_{2},v_{2}).
\]
In particular,
\[
-|E_{1}|^2=-|E_{2}|^2=|E_{3}|^2=|E_{4}|^2=1,\qquad G(E_i,E_j)=0,\;\;\forall i\neq i.
\] 
A brief computation gives
\[
Ric^-(E_{1},E_{1})=Ric^-(E_2,E_2)=\kappa(g_1)(x)+2\kappa(g_2)(y),
\]
\[
Ric^-(E_{3},E_{3})=Ric^-(E_4,E_4)=2\kappa(g_1)(x)+\kappa(g_2)(y),
\]
where $R^-$ is the Ricci tensor of the metric $G^-$. Then, if $R^-$ denotes the scalar curvature of $G^-$, we have
\begin{eqnarray}
R^-&=&\sum_{k=1}^2\Big(-Ric^-(E_{k},E_{k})+Ric^-(E_{2+k},E_{2+k})\Big)\nonumber \\
&=&2(\kappa(g_1)(x)-\kappa(g_2)(y)).\label{e:scnegat}
\end{eqnarray}
If the neutral K\"ahler metric $G^-$ is conformally flat, it is also scalar flat \cite{AR} and hence, from (\ref{e:scnegat}), the Gauss curvatures $\kappa(g_1)$ and $\kappa(g_2)$ are constants with $\kappa(g_1)=\kappa(g_2)$.

Following the same argument as before, assume the converse. That is, $\kappa(g_1)=\kappa(g_2)=c$, where $c$ is a real constant. Consider the corresponding coframe ${\cal B}_2=(e_1,e_2,e_3,e_4)$ and the Hodge star operator $\ast:\Lambda^2(\Sigma_1\times\Sigma_2)\rightarrow \Lambda^2(\Sigma_1\times\Sigma_2)$. 

The Hodge star operator splits the Riemann curvature operator ${\cal R}:\Lambda^2(\Sigma_1\times\Sigma_2)\rightarrow \Lambda^2(\Sigma_1\times\Sigma_2)$ in the same way as in the Riemannian case. The Weyl $(0,4)$-tensor $W$ is given by:
\[
W_{ijkl}=R^G_{ijkl}-\frac{1}{2}(-G_{jk}Ric^G_{il}+G_{jl}Ric^G_{ik}-G_{il}Ric^G_{jk}+G_{ik}Ric^G_{jl}),
\]
where $R^G_{ijkl}=G(R^G(E_i,E_j)E_k,E_l)$. An orthonormal basis for $\Lambda^2_{\pm}(\Sigma_1\times\Sigma_2)$, in the neutral case, is
\begin{eqnarray}
e^{\pm}_1&=&\frac{1}{\sqrt{2}}(e_1\wedge e_2\pm e_3\wedge e_4),\nonumber \\
e^{\pm}_2&=&\frac{1}{\sqrt{2}}(e_1\wedge e_3\pm e_2\wedge e_4),\nonumber \\
e^{\pm}_3&=&\frac{1}{\sqrt{2}}(e_1\wedge e_4\mp e_2\wedge e_3).\nonumber 
\end{eqnarray}
The fact that $G^-$ is scalar flat, following \cite{AR}, the anti-self-dual part $W^{-}$ vanishes since,
\[
W^{-}=R^-\begin{pmatrix} 1/3 & & \\ & 1/6 & \\ & & 1/6\end{pmatrix}.
\]
The self-dual part is
\[
W^{+}=\begin{pmatrix} W_{1212}+W_{3434}+2W_{1234} & 2(W_{1213}+W_{1334}) & 2(W_{1214}+W_{1434}) \\ & 2(W_{1313}+W_{1324}) & 2(W_{1314}-W_{1323})\\ & &  2(W_{1414}-W_{1423})\end{pmatrix},
\]
and a brief computation shows that $W^{+}$ vanishes. Therefore, the Weyl tensor $W$ vanishes, or $G$ is locally conformally flat.
\end{proof}

\vspace{0.1in}

\begin{coro}
Let $(\Sigma,g)$ be a Riemannian two manifold. The neutral K\"ahler metric $G^-$ of the four dimensional K\"ahler manifold $\Sigma\times\Sigma$ is conformally flat if and only if the metric $g$ is of constant Gaussian curvature.
\end{coro}

\vspace{0.2in}

\section{Surface theory in the 4-manifold $\Sigma_1\times\Sigma_2$ }\label{s:rankoneranotwo}

Let $\Phi: S\rightarrow\Sigma_1\times\Sigma_2$ be a smooth immersion of a surface $S$ in $\Sigma_1\times\Sigma_2$, where $(\Sigma_1,g_1)$ and $(\Sigma_2,g_2)$ are both Riemannian two manifolds and let $\pi_i$ be the projections of $\Sigma_1\times\Sigma_2$ onto $\Sigma_i$, $i=1,2$. We denote by $\phi$ and $\psi$ the mappings $\pi_1\circ\Phi$ and $\pi_2\circ\Phi$, respectively and we write $\Phi=(\phi,\psi)$. The rank of a mapping at a point is the rank of its derivative at that point.
\begin{defi}\label{d:defiprojectedrank}
The immersion $\Phi=(\phi,\psi): S\rightarrow\Sigma_1\times\Sigma_2$ is said to be of \emph{projected rank zero} at a point $p\in S$ if either $\mbox{rank}(\phi(p))=0$ or $\mbox{rank}(\psi(p))=0$. $\Phi$ is of \emph{projected rank one} at $p$ if either $\mbox{rank}(\phi(p))=1$ or $\mbox{rank}(\psi(p))=1$. Finally, $\Phi$ is of \emph{projected rank two} at $p$ if $\mbox{rank}(\phi(p))=\mbox{rank}(\psi(p))=2$. 
\end{defi}
Note that, since it is an immersion, $\Phi$ must be of projected rank zero, one or two.

\vspace{0.1in}

\subsection{Projected rank zero case}

Let $\Phi=(\phi,\psi)$ be of projected rank zero immersion in $\Sigma_1\times\Sigma_2$. Assuming, without loss of generality, that $\mbox{rank}(\phi)=0$, the map $\phi$ is locally a constant function and the map $\psi$ is a local diffeomorphism.  We now give the following Proposition:
\begin{prop}
There are no Lagrangian immersions in $\Sigma_1\times\Sigma_2$ of projected rank zero.
\end{prop}
\begin{proof}
If $\Phi=(\phi,\psi): S\rightarrow\Sigma_1\times\Sigma_2$ were an immersed surface with $rank(\phi)=0$, then $\psi:S\rightarrow \Sigma_2$ is a 
local diffeomorphism and thus for any vector fields $X,Y$ on $S$ 
\begin{eqnarray}
\Phi^{\ast}\Omega^{\epsilon}(X,Y)&=&\Omega^{\epsilon}(d\Phi(X),d\Phi(Y))\nonumber \\
&=&\Omega^{\epsilon}((0,d\psi(X)),(0,d\psi(Y)))\nonumber \\
&=&\epsilon\;\omega(d\psi(X),d\psi(Y))\nonumber \\
&\neq&0,\nonumber
\end{eqnarray}
where the last line follows from the non-degeneracy of $\omega$ and the fact that $d\psi$ is a bundle isomorphism.
\end{proof}

\vspace{0.1in}

\subsection{Projected rank one Lagrangian surfaces}

We begin by giving a definition of the Cornu spirals in a Riemannian two manifold.

\begin{defi}
Let $(\Sigma,g)$ be a Riemannian two manifold. A regular curve $\gamma$ of $\Sigma$ is called a \emph{Cornu spiral of parameter $\lambda$} if its curvature $\kappa_{\gamma}$ is a linear function of its arclength parameter such that $\kappa_{\gamma}(s)=\lambda s+\mu$, where $s$ is the arclength and $\lambda,\mu$ are real constants.  
\end{defi}

A Cornu spiral $\gamma$ in ${\mathbb R}^2$ of parameter $\lambda$ can be parametrised, up to congruences, by 
\[
\gamma(s)=\left(\int_0^s\cos(\lambda t^2/2)dt,\int_0^s\sin(\lambda t^2/2)dt\right),
\]
and they are bounded but have infinite length \cite{AnCa}.

Let $\Phi=(\phi,\psi):S\rightarrow\Sigma_1\times\Sigma_2$ be of projected rank one immersion in $\Sigma_1\times\Sigma_2$. Then either $\phi$ or $\psi$ is of rank one. The following theorem gives all rank one Hamiltonian $G^{\epsilon}$-minimal surfaces:

\begin{theo}\label{t:rankonethe}
Every projected rank one Lagrangian surface can be locally parametrised by $\Phi:S\rightarrow \Sigma_1\times\Sigma_2:(s,t)\mapsto (\phi(s),\psi(t))$, where $\phi$ and $\psi$ are regular curves on $\Sigma$ and the induced metric $\Phi^{\ast}G^{\epsilon}$ is flat. In addition, $\Phi$ is Hamiltonian $G^{\epsilon}$-minimal if and only if $\phi$ and $\psi$ are Cornu spirals of parameters $\lambda_{\phi}$ and $\lambda_{\psi}$ respectively such that
\[
\lambda_{\phi}=-\epsilon\lambda_{\psi}.
\]
Moreover, $\Phi$ is $G^{\epsilon}$-minimal Lagrangian if and only if  both $\phi$ and $\psi$ are geodesics, and every projected rank one $G^{\epsilon}$-minimal Lagrangian surface in $\Sigma_1\times\Sigma_2$ is totally geodesic.
\end{theo} 
\begin{proof}
Let $\Phi=(\phi,\psi):S\rightarrow\Sigma_1\times\Sigma_2$ be of  projected rank one Lagrangian immersion. Assume, without loss of generality, that $\phi$ is of rank one. We now prove that $\psi$ is of rank one. 

Since $\Phi$ is an immersion of a surface, the map $\psi$ cannot be of rank zero. Suppose that $\psi$ is of rank two, i.e. a local diffeomorphism. 
Thus, $\Phi$ is locally parametrised by $\Phi:U\subset S\rightarrow\Sigma_1\times\Sigma_2:(s,t)\mapsto (\phi(s),\psi(s,t))$. Hence,
\[
\Phi_s=(\phi'(s),\psi_s)\qquad \Phi_t=(0,\psi_t).
\]
Since $\Phi$ is a Lagrangian immersion we have that $\omega_2(\psi_s,\psi_t)=0$. 

The fact that $\psi$ is a local diffeomorphism implies that for any non-zera vector field $X$ in $\Sigma_2$ can be written as $X=a\psi_s+b\psi_t$ and therefore we have that $\omega_2(\psi_s,X)=0$. The nondegeneracy of $\omega_2$ implies that $\psi$ is cannot be a local diffeomorphism  since $\psi_s=0$. Thus 
$\psi$ is also a rank one immersion.

We now have that $S$ is locally parametrised by $\Phi:U\subset S\rightarrow\Sigma_1\times\Sigma_2:(s,t)\mapsto (\phi (s),\psi (t))$, where $\phi$ and $\psi$ are regular curves in $\Sigma_1$ and $\Sigma_2$, respectively. If $s,t$ are the corresponding arc-length parameters of $\phi$ and $\psi$, the Fr\'enet equtions give
\[
D^1_{\phi'}\phi'=k_{\phi}j\phi'\qquad D^2_{\psi'}\psi'=k_{\psi}j\psi',
\]
where $k_{\phi}$ and $k_{\psi}$ denote the curvatures of $\phi$ and $\psi$, respectively. Moreover, $\Phi_s=(\phi',0)$ and $\Phi_t=(0,\psi')$ and thus,
\[
\nabla_{\Phi_s}\Phi_s=(D^1_{\phi'}\phi',0)=(k_{\phi}j\phi',0),\quad \nabla_{\Phi_t}\Phi_t=(0,D^2_{\psi'}\psi')=(0,k_{\psi}j\psi'),\quad \nabla_{\Phi_t}\Phi_s=(0,0).
\] 
The first fundamental form $G^{\epsilon}_{ij}=G^{\epsilon}(\partial_i\Phi,\partial_j\Phi)$ is given by
\[
G_{ss}=\epsilon G_{tt}=1,\qquad G_{st}=0,
\]
which proves that the immersion $\Phi$ is flat. 

The second fundamental form $h^{\epsilon}$ of $\Phi$, is completely determined by the following tri-symmetric tensor 
\[
h^{\epsilon}(X,Y,Z):=G^{\epsilon}(h^{\epsilon}(X,Y),JZ)=\Omega^{\epsilon}(X,\nabla_Y Z).
\]
We then have,
\[
h^{\epsilon}_{sst}=\Omega^{\epsilon}(\Phi_s,\nabla_{\Phi_s}\Phi_t)=0,\quad h^{\epsilon}_{stt}=\Omega^{\epsilon}(\Phi_s,\nabla_{\Phi_t}\Phi_t)=0.
\]
Moreover,
\[
h_{sss}^{\epsilon}=\Omega^{\epsilon}(\Phi_s,\nabla_{\Phi_s}\Phi_s)=\Omega^{\epsilon}((\phi',0),(k_{\phi}j\phi',0))=G^{\epsilon}((j\phi',0),(k_{\phi}j\phi',0))=k_{\phi},
\]
and similarly, $h_{ttt}^{\epsilon}=\epsilon k_{\psi}$. Denote the mean curvature of $\Phi$ with respect to the metric $G^{\epsilon}$ by 
$\vec{H}^{\epsilon}$. Then
\[
G^{\epsilon}(2\vec{H}^{\epsilon},J\Phi_s)=\frac{h^{\epsilon}_{sss}G^{\epsilon}_{tt}+h^{\epsilon}_{stt}G^{\epsilon}_{ss}-
2h^{\epsilon}_{sst}G^{\epsilon}_{st}}{G^{\epsilon}_{ss}G^{\epsilon}_{tt}-(G^{\epsilon}_{st})^2}=k_{\phi},
\]
and
\[
G^{\epsilon}(2\vec{H}^{\epsilon},J\Phi_t)=\frac{h^{\epsilon}_{sst}G^{\epsilon}_{tt}+
h^{\epsilon}_{ttt}G^{\epsilon}_{ss}-2h^{\epsilon}_{stt}G^{\epsilon}_{st}}{G^{\epsilon}_{ss}G^{\epsilon}_{tt}-(G^{\epsilon}_{st})^2}=k_{\psi}.
\] 
Hence
\[
2\vec{H}^{\epsilon}=k_{\phi}J\Phi_s+\epsilon k_{\psi}J\Phi_t.
\]
It is not hard to see that the Lagrangian immersion $\Phi$ is $G^{\epsilon}$-minimal iff the curves $\phi$ and $\psi$ are geodesics. Moreover, if $\Phi$ is $G^{\epsilon}$-minimal Lagrangian it is totally geodesic since the second fundamental form vanishes identically. 

Note also that,
\[
\mbox{div}^{\epsilon}(\Phi_s)=-G^{\epsilon}(\nabla_{\Phi_s}\Phi_s,\Phi_s)=-G^{\epsilon}((k_{\phi}j\phi',0),(\phi',0))=-g(k_{\phi}j\phi',\phi')= 0.
\]
In a similar way, we derive that $\mbox{div}^{\epsilon}(\Phi_t)=0$.

Thus,
\begin{eqnarray}
-\mbox{div}^{\epsilon}(2J\vec{H}^{\epsilon})&=&G^{\epsilon}(\nabla k_{\phi},\Phi_s)+k_{\phi}\mbox{div}^{\epsilon}(\Phi_s)+\epsilon G^{\epsilon}(\nabla k_{\psi},\Phi_t)+\epsilon k_{\psi}\mbox{div}^{\epsilon}(\Phi_t)\nonumber\\
&=& \frac{D}{ds}k_{\phi}(s)+\epsilon \frac{D}{dt}k_{\psi}(t),\nonumber
\end{eqnarray}
and the theorem follows. 
\end{proof}

\vspace{0.1in}

\subsection{Projected rank two Lagrangian surfaces}
For the projected rank two case, we have the following Theorem:

\begin{theo}\label{t:notflattt} Let $(\Sigma_1,g_1)$ and $(\Sigma_2,g_2)$ be Riemannian two manifolds and let $(G^{\epsilon},J,\Omega^{\epsilon})$ be the canonical K\"ahler product structures on $\Sigma_1\times\Sigma_2$ constructed in section \ref{s:construction}. Let $\kappa(g_1),\kappa(g_2)$ be the Gauss curvatures  of $g_1$ and $g_2$ respectively. Assume that one of the following holds:

\noindent \emph{(i)} The metrics $g_1$ and $g_2$ are both generically non-flat and $\epsilon \kappa(g_1)\kappa(g_2)<0$ away from flat points.

\noindent \emph{(ii)} Only one of the metrics $g_1$ and $g_2$ is flat while the other is non-flat generically. 

Then every $G^{\epsilon}$-minimal Lagrangian surface is of  projected rank one.
\end{theo}

\begin{proof}
Assume that the $G^{\epsilon}$-minimal Lagrangian immersion $\Phi=(\phi,\psi):S\rightarrow \Sigma_1\times\Sigma_2$ is of  projected rank two. Then by definition the  mappings $\phi:S\rightarrow \Sigma_1$ and $\psi:S\rightarrow \Sigma_2$ are both local diffeomorphisms. The Lagrangian assumption $\Phi^{\ast}\Omega^{\epsilon}=0$ yields 
\begin{equation}\label{e:lagpositive}
\phi^{\ast}\omega_1=-\epsilon\psi^{\ast}\omega_2.
\end{equation}
Take an orthonormal frame $(e_1,e_2)$ of $\Phi^{\ast}G^{\epsilon}$ such that,
\[
G^{\epsilon}(d\Phi(e_1),d\Phi(e_1))=\epsilon G^{\epsilon}(d\Phi(e_2),d\Phi(e_2))=1,\qquad G^{\epsilon}(d\Phi(e_1),d\Phi(e_2))=0.
\]
The Lagrangian condition implies that the frame $(d\Phi(e_1),d\Phi(e_2),Jd\Phi(e_1),Jd\Phi(e_2))$ is orthonormal. Let $(s_1,s_2)$ and $(v_1,v_2)$ be oriented orthonormal frames of $(\Sigma_1,g_1)$ and $(\Sigma_2,g_2)$ respectively such that $j_1s_1=s_2$ and $j_2v_1=v_2$. Then there exist smooth functions $\lambda_1,\lambda_2,\mu_1,\mu_2$ on $\Sigma_1$ and $\bar\lambda_1,\bar\lambda_2,\bar\mu_1,\bar\mu_2$ on $\Sigma_2$ such that
\[
d\phi(e_1)=\lambda_1 s_1+\lambda_2 s_2\qquad d\phi(e_2)=\mu_1 s_1+\mu_2 s_2,
\]
\[
d\psi(e_1)=\bar\lambda_1 v_1+\bar\lambda_2 v_2\qquad d\psi(e_2)=\bar\mu_1 v_1+\bar\mu_2 v_2.
\]
Thus,
\[
\phi^{\ast}\omega_1(e_1,e_2)=\lambda_1\mu_2-\lambda_2\mu_1,\qquad \psi^{\ast}\omega_2(e_1,e_2)=\bar\lambda_1\bar\mu_2-\bar\lambda_2\bar\mu_1,
\]
and using the Lagrangian condition (\ref{e:lagpositive}) we have
\[
(\lambda_1\mu_2-\lambda_2\mu_1)(\phi(p))=-\epsilon(\bar\lambda_1\bar\mu_2-\bar\lambda_2\bar\mu_1)(\psi(p)),\quad\forall\;\; p\in S.
\]
Moreover, the assumption that $\Phi$ is of projected rank two, implies that $\lambda_1\mu_2-\lambda_2\mu_1\neq 0$ for every $p\in S$.

If $H^{\epsilon}$ is the mean curvature vector of the immersion $\Phi$, consider the one form $a_{H^{\epsilon}}$ defined by $a_{H^{\epsilon}}=G^{\epsilon}(JH^{\epsilon},\cdot)$. Since $\Phi$ is Lagrangian, it is known from \cite{CTU} that
\begin{equation}\label{e:ricciform}
da_{H^{\epsilon}}=\Phi^{\ast}\rho^{\epsilon},
\end{equation} 
where $\rho^{\epsilon}$ is the Ricci form of $G^{\epsilon}$. The fact that $\Phi$ is a $G^{\epsilon}$-minimal Lagrangian immersion implies that $\Phi^{\ast}\rho^{\epsilon}$ vanishes and therefore,
\begin{eqnarray}
0&=&\rho^{\epsilon}(d\Phi(e_1),d\Phi(e_2))\nonumber \\
&=& Ric^{\epsilon}(d\Phi(e_1),Jd\Phi(e_2))\nonumber \\
&=& \epsilon G^{\epsilon}(R(d\Phi e_1,d\Phi e_2)Jd\Phi e_2,d\Phi e_2)+G^{\epsilon}(R(d\Phi e_1 ,d\Phi e_2)Jd\Phi e_1,d\Phi e_1)\nonumber\\
&=& \epsilon g_1(R_1(d\phi e_1,d\phi e_2)j_1d\phi e_2,d\phi e_2)+ g_2(R_2(d\psi e_1 ,d\psi e_2)j_2d\psi e_2,d\psi e_2)\nonumber \\
&&\qquad\quad  +g_1(R_1(d\phi e_1,d\phi e_2)j_1d\phi e_1,d\phi e_1)+\epsilon g_2(R_2(d\psi e_1 ,d\psi e_2)Jd\psi e_1,d\psi e_1).\nonumber \\
&=& \epsilon \Big((\lambda_1^2+\lambda_2^2+\epsilon(\mu_1^2+\mu_2^2)\Big)(\mu_1\lambda_2-\mu_2\lambda_1)\kappa(g_1)\nonumber \\
&&\qquad\quad  +\Big(\bar\lambda_1^2+\bar\lambda_2^2+\epsilon(\bar\mu_1^2+\bar\mu_2^2)\Big)(\bar\mu_1\bar\lambda_2-\bar\mu_2\bar\lambda_1)\kappa(g_2)\nonumber \\
&=& \epsilon (\mu_1\lambda_2-\mu_2\lambda_1)\Big[\Big(\lambda_1^2+\lambda_2^2+\epsilon(\mu_1^2+\mu_2^2)\Big)\kappa(g_1)-
\Big(\bar\lambda_1^2+\bar\lambda_2^2+\epsilon(\bar\mu_1^2+\bar\mu_2^2)\Big)\kappa(g_2)\Big)\Big]\nonumber 
\end{eqnarray} 
which finally gives,
\begin{equation}\label{e:consd}
\Big(\lambda_1^2+\lambda_2^2+\epsilon(\mu_1^2+\mu_2^2)\Big)\kappa(g_1)=
\Big(\bar\lambda_1^2+\bar\lambda_2^2+\epsilon(\bar\mu_1^2+\bar\mu_2^2)\Big)\kappa(g_2).
\end{equation}
The condition $G^{\epsilon}(d\Phi(e_1),d\Phi(e_2))=0$ yields
\begin{equation}\label{e:codn}
\lambda_1\mu_1+\lambda_2\mu_2=-\epsilon(\bar\lambda_1\bar\mu_1+\bar\lambda_2\bar\mu_2).
\end{equation}
Now using (\ref{e:lagpositive}) and (\ref{e:codn}) we have
\begin{equation}\label{e:codn1}
(\lambda_1^2+\lambda_2^2)(\mu_1^2+\mu_2^2)=(\bar\lambda_1^2+\bar\lambda_2^2)(\bar\mu_1^2+\bar\mu_2^2).
\end{equation}
From $G^{\epsilon}(d\Phi(e_1),d\Phi(e_1))=\epsilon G^{\epsilon}(d\Phi(e_2),d\Phi(e_2))=1$ we obtain
\begin{equation}\label{e:codn2}
\lambda_1^2+\lambda_2^2+\epsilon(\bar\lambda_1^2+\bar\lambda_2^2)=\epsilon(\mu_1^2+\mu_2^2)+\bar\mu_1^2+\bar\mu_2^2=1.
\end{equation}
Set $a:=\lambda_1^2+\lambda_2^2$, $b:=\mu_1^2+\mu_2^2$, $\bar a:=\bar\lambda_1^2+\bar\lambda_2^2$, $\bar b:=\bar\mu_1^2+\bar\mu_2^2$. The relations (\ref{e:codn}), (\ref{e:codn1}) and (\ref{e:codn2}) give
\[
ab=\bar a\bar b\qquad a+\epsilon\bar a=\epsilon b+\bar b=1.
\]
Thus $a=-\epsilon\bar a+1$ and $b=\epsilon-\epsilon\bar b$, and from $ab=\bar a\bar b$ we have that $\bar a+\epsilon\bar b=\epsilon$. Moreover, $\bar a=\epsilon-\epsilon a$ and $\bar b=1-\epsilon b$, and again from $ab=\bar a\bar b$ we have $a+\epsilon b=1$. Hence, relation (\ref{e:consd}) becomes
\[
\kappa(g_1)(\phi(p))=\epsilon \kappa(g_2)(\psi(p)),\qquad \mbox{for}\;\mbox{every}\;\; p\in S,
\]
which implies that the metrics $g_1$ and $g_2$ can satisfy neither condition (i) nor condition (ii) of the statement. 
\end{proof}

\vspace{0.1in}

The following Corollaries follow:

\begin{coro}
Every $G^+$-minimal Lagrangian surface immersed in ${\mathbb S}^2\times {\mathbb H}^2$ is, up to isometry, the cylinder ${\mathbb S}^1\times {\mathbb R}$. Moreover, every $G^{\epsilon}$-minimal Lagrangian surface immersed in ${\mathbb R}^2\times {\mathbb H}^2$ \emph{(${\mathbb R}^2\times {\mathbb S}^2$)} is of projected rank one and therefore it is $\gamma_1\times\gamma_2$, where $\gamma_1$ is a straight line in ${\mathbb R}^2$  and $\gamma_2$ is a geodesic in ${\mathbb H}^2$  \emph{($\gamma_2$ is a geodesic in ${\mathbb S}^2$)}, respectively.
\end{coro}

\begin{coro}
Let $(\Sigma,g)$ be a Riemannian two manifold such that the metric $g$ is non-flat. Then every $G^{-}$-minimal Lagrangian surface immersed in $\Sigma\times\Sigma$ is of projected rank one and is therefore the product of two geodesics of $(\Sigma,g)$.
\end{coro}

\vspace{0.2in}

\section{The Hamiltonian stability of minimal Lagrangian surfaces}\label{s:hamiltstabilitysection}

The Hamiltonian stability of a Hamiltonian minimal surface $S$ in a pseudo-Riemannian manifold $(\M, G)$ is given by the monotonicity of the second variation formula of the volume $V(S)$ under Hamiltonian deformations (see \cite{Oh1} and \cite{AnGer}). For a smooth compactly supported function  $u\in C^{\infty}_{c}(S)$ the second variation $\delta^2 V(S)(X)$ formula in the direction of the Hamiltonian vector field $X=J\nabla u$ is:
\[
\delta^2 V(S)(X)=\int_{S} \Big((\Delta u)^2 -  Ric^{G}(\nabla u,\nabla u)-2G(h(\nabla u,\nabla u),nH)+G^2(nH,J\nabla u)\Big)dV,
\]
where $h$ is the second fundamental form of $S$, $Ric^G$ is the Ricci curvature tensor of the metric $G$, and $\Delta$ with $\nabla$ denote the Laplacian and gradient, respectively, with respect to the metric $G$ induced on $S$. For the Hamiltonian stability of projected rank one Hamiltonian $G^{\epsilon}$-minimal surfaces we give the following Theorem:
\begin{theo}\label{t:hstability}
Let $\Phi=(\phi,\psi)$ be of projected rank one Hamiltonian $G^{\epsilon}$-minimal immersion in $(\Sigma_1\times\Sigma_2,G^{\epsilon})$ such that $\kappa(g_1)\leq -2k_{\phi}^2$ and $\kappa(g_2)\leq -2k_{\psi}^2$ along the curves $\phi$ and $\psi$ respectively. Then $\Phi$ is a  local minimizer of the volume in its Hamiltonian isotopy class.
\end{theo}
\begin{proof}
Let $\Phi=(\phi,\psi):S\rightarrow \Sigma_1\times\Sigma_2$ be of projected rank one Hamiltonian $G^{\epsilon}$-minimal immersion and let $(s,t)$ be the corresponded arclengths of $\phi$ and $\psi$, respectively.

Then $(\phi_s,j_1\phi_s)$ is an oriented orthonormal frame of $(\Sigma_1,g_1)$ and $(\psi_t,j_2\psi_t)$ is an oriented orthonormal frame of $(\Sigma_2,g_2)$.

Therefore,
\begin{eqnarray}
Ric^{\epsilon}(\Phi_s,\Phi_s)&=&\epsilon G^{\epsilon}(R( \Phi_t,\Phi_s)\Phi_s,\Phi_t)+G^{\epsilon}(R( J\Phi_s,\Phi_s)\Phi_s,J\Phi_s)\nonumber \\
&&\qquad\qquad\qquad\qquad \qquad\qquad +\epsilon G^{\epsilon}(R( J\Phi_t,\Phi_s)\Phi_s,J\Phi_t)\nonumber \\
&=&G^{\epsilon}(R( J\Phi_s,\Phi_s)\Phi_s,J\Phi_s)\nonumber \\
&=&G^{\epsilon}((R_1( j_1\phi_s,\phi_s)\phi_s,R_2( j_2\psi_s,\psi_s)\psi_s),(j_1\phi_s,j_2\psi_s))\nonumber \\
&=&G^{\epsilon}((R_1( j_1\phi_s,\phi_s)\phi_s,0),(j_1\phi_s,0))\nonumber \\
&=&g_1(R_1( j_1\phi_s,\phi_s)\phi_s,j_1\phi_s)\nonumber \\
&=&\kappa(g_1).\nonumber
\end{eqnarray}
Moreover, a similar computation gives
\[
Ric^{\epsilon}(\Phi_t,\Phi_t)=\kappa(g_2)\quad\mbox{and}\quad Ric^{\epsilon}(\Phi_s,\Phi_t)=0.
\]
Then, for every $u(s,t)\in C^{\infty}_c(S)$ we have
\[
Ric^{\epsilon}(\nabla u,\nabla u)=\kappa(g_1)u_s^2+\kappa(g_2)u_t^2.
\]
Furthermore 
\[
G^{\epsilon}(h^{\epsilon}(\nabla u,\nabla u),2\vec{H}^{\epsilon})=u_s^2k_{\phi}^2+u_t^2k_{\psi}^2,
\] 
and 
\[
G^{\epsilon}(2\vec{H}^{\epsilon},J\nabla u)=u_s k_{\phi}+\epsilon u_t k_{\psi}.
\]
The second variation formula for the volume functional with respect of the Hamiltonian vector field $X=J\nabla u$ therefore becomes
\begin{eqnarray}
\delta^2 V(S)(X)&=&\int_{S} (\Delta^{\epsilon} u)^2 -  Ric^{\epsilon}(\nabla u,\nabla u)-2G^{\epsilon}(h^{\epsilon}(\nabla u,\nabla u),2\vec{H}^{\epsilon})+G^{\epsilon}(2\vec{H}^{\epsilon},J\nabla u)^2\nonumber \\
&=&\int_{S} (u_{ss}+\epsilon u_{tt})^2-u_s^2\kappa(g_1)-u_t^2\kappa(g_2)-(u_s k_{\phi}-\epsilon u_t k_{\psi})^2\nonumber \\
&=&\int_{S} (u_{ss}+\epsilon u_{tt})^2+u_s^2(-\kappa(g_1)-k_{\phi}^2)+u_t^2(-\kappa(g_2)-k_{\psi}^2)+2\epsilon u_s u_t k_{\phi}k_{\psi}.\nonumber
\end{eqnarray}
Assuming that $\kappa(g_1)\leq -2k_{\phi}^2$ and $\kappa(g_2)\leq -2k_{\psi}^2$ along the curves $\phi$ and $\psi$, respectively we conclude that the second variation formula is nonnegative.
\end{proof}
 
Every minimal Lagrangian surface in a pseudo-K\"ahler 4-manifold is unstable \cite{An2}. The following Corollary explores the Hamiltonian stability of $G^-$-minimal Lagrangian surfaces in $\Sigma_1\times\Sigma_2$:
\begin{coro}
Let $(\Sigma_1,g_1)$ and $(\Sigma_2,g_2)$ be Riemannian two manifolds such that their Gauss curvatures $\kappa(g_1)$ and $\kappa(g_2)$ are both negative. Then every $G^-$-minimal Lagrangian surface is a  local minimizer of the volume in its Hamiltonian isotopy class.
\end{coro} 
\begin{proof}
From Theorem \ref{t:notflattt} every $G^-$-minimal Lagrangian immersion must be of projected rank one and thus it is parametrised by $\Phi=(\phi,\psi):S\rightarrow \Sigma_1\times\Sigma_2$, where $\phi=\phi(s)$ and $\psi=\psi(t)$, where $s,t$ are arclengths. Assuming that $\kappa(g_1),\kappa(g_2)$ are both negative we have that:
\[
\kappa(g_1)(s)\leq -2k_{\phi}^2(s)=0,\qquad \kappa(g_2)(t)\leq -2k_{\psi}^2(t)=0,
\]
and therefore from Theorem \ref{t:hstability} the $G^-$-minimal Lagrangian immersion $\Phi$ is stable under Hamiltonian deformations.
\end{proof}

We also have the Corollary:

\begin{coro}
Let $(\Sigma,g)$ be a Riemannian two manifold of negative Gaussian curvature. Then every $G^-$-minimal Lagrangian surface immersed in $\Sigma\times\Sigma$ is a local minimizer of the volume in its Hamiltonian isotopy class.
\end{coro} 

\noindent
{\bf Example 1.} It is easy to see that if $(\Sigma,g)$ is a Riemannian two manifold of constant Gauss curvature $c\neq 0$, then every $G^-$-minimal Lagrangian surface immersed in $\Sigma\times\Sigma$ is a local minimizer of the volume in its Hamiltonian isotopy class if and only if $c<0$.  

\vspace{0.1in}

\noindent
{\bf Example 2.} Let $L({\mathbb S}^3)$ and $L^{+}(Ad{\mathbb S}^3)$ be the spaces of oriented closed geodesics in the three sphere and anti-De Sitter 3-space, respectively. Then  $L({\mathbb S}^3)={\mathbb S}^2\times {\mathbb S}^2$ and $L^{+}(Ad{\mathbb S}^3)={\mathbb H}^2\times {\mathbb H}^2$ (see \cite{AGK} and \cite{An4}). 
The previous example generalises a result obtained in \cite{AnGer} which states that every minimal Lagrangian surface in the space of closed oriented geodesics $L({\mathbb S}^3)$ is Hamiltonian unstable and every Lagrangian minimal surface in $L^{+}(Ad{\mathbb S}^3)$ is Hamiltonian stable.

\vspace{0.1in}

The following Proposition investigates the Hamiltonian stability of $G^+$-minimal Lagrangian surfaces: 
\begin{prop}
Let $(\Sigma_1,g_1)$ and $(\Sigma_2,g_2)$ be Riemannian two manifolds with Gaussian curvatures bounded in the following way:
\[
c_1\leq |\kappa(g_1)(x)|\leq C_1,\quad c_2\leq |\kappa(g_2)(y)|\leq C_2,\quad\mbox{and}\quad \kappa(g_1)(x)\kappa(g_2)(y)<0,
\]
for every pair $(x,y)\in \Sigma_1\times\Sigma_2$ and for some positive constants $c_1,c_2,C_1,C_2$. Then, every $G^+$-minimal Lagrangian surface is Hamiltonian unstable and therefore is $G^+$-unstable.
\end{prop} 
\begin{proof}
Consider again a Lagrangian minimal immersion $\Phi=(\phi,\psi):S\rightarrow \Sigma_1\times\Sigma_2$. From Theorem \ref{t:notflattt}, we have that $\phi=\phi(s)$ and $\psi=\psi(t)$ are geodesics of $\Sigma_1$ and $\Sigma_2$, respectively, with $(s,t)$ chosen to be the corresponding arc-lengths.

Then $(\phi_s,j_1\phi_s)$ is an oriented orthonormal frame of $(\Sigma_1,g_1)$ and $(\psi_t,j_2\psi_t)$ is an oriented orthonormal frame of $(\Sigma_2,g_2)$.

A similar computation as in Theorem \ref{t:hstability} gives,
\[
Ric^+(\Phi_s,\Phi_s)=\kappa(g_1),\qquad Ric^+(\Phi_t,\Phi_t)=\kappa(g_2),\qquad Ric^+(\Phi_s,\Phi_t)=0,
\]
and therefore, the second variation formula for the volume of $S$ in the direction of the Hamiltonian vector field $X=J\nabla u$ is
\[
\delta^2 V(S)(X)=\int_{S} \Big((u_{ss}-u_{tt})^2-\kappa(g_1)u_s^2-\kappa(g_2)u_t^2\Big)dV.
\]
Assume that $\kappa(g_1)<0$. Then, $\kappa(g_2)>0$ and
\[
\delta^2 V(S)(X)\geq \int_{S} \Big((u_{ss}-u_{tt})^2-C_1u_s^2+c_2u_t^2\Big)dV.
\]
Thus, for the quadratic functional 
\[
Q_1(u):=\int_{S}-C_1u_s^2+c_2u_t^2,
\] 
there exists $u^1\in C^{\infty}_{c}(S)$ such that $Q_1(u^1)\geq 0$. Therefore $\delta^2 V(S)(J\nabla u^1)\geq 0$.

On the other hand, for every $u\in C^{\infty}_{c}(S)$
\[
\delta^2 V(S)(J\nabla u)\leq \int_{S} \Big((u_{ss}+u_{tt})^2-c_1u_s^2+C_2u_t^2\Big)dV.
\]
Then for the quadratic functional
\[
Q_2(u):=\int_{S}-c_1u_s^2+C_2u_t^2,
\] 
there exists $u^2\in C^{\infty}_{c}(S)$ such that $Q_2(u^2)\leq 0$. A similar argument as in the proof of Theorem 3 of \cite{AnGer} establishes the existence of $u^3\in C^{\infty}_{c}(S)$ such that 
\[
\int_{S} \Big((u^3_{ss}+u^3_{tt})^2-c_1(u^3_s)^2+C_2(u^3_t)^2\Big)dV\leq 0,
\]
which implies that $\delta^2 V(S)(J\nabla u^3)\leq 0$ and therefore the second variation formula for the volume of $S$ under Hamiltonian deformations is indefinite.
\end{proof}

\section*{Appendix}

The equation (\ref{e:ricciform}) has been proved in \cite{Da} (see also \cite{Br} for an alternative proof of it) for the case of a Lagrangian submanifold immersed in a Riemannian K\"ahler manifold. In the Appendix, we show that the same equation holds true for Lagrangian immersions in a neutral K\"aher 4-manifold.

\medskip

Suppose now that $\Phi: S\rightarrow M^4$ is a smooth Lagrangian immersion of a surface $S$ in a neutral K\"ahler 4-manifold $(M,J,G,\Omega)$. Let $(e_1,e_2)$ be an orthonormal frame of the induced metric $\Phi^{\ast}G$ such that $|e_1|^2=-|e_2|^2=1$. The Lagrangian condition implies that $(e_1,e_2,Je_1,Je_2)$ is an orthonormal frame of $G$. If $h$ is the second fundamental form of $\Phi$ and $H=\frac{1}{2}(h(e_1,e_1)-h(e_2,e_2))$ is the mean curvature vector, we consider the Maslov 1-form $a_{H}$ of $S$, defined by
\[
a_H:=G(JH,.),
\]
\begin{prop}
Let $\Phi$ be a Lagrangian immersion in a neutral K\"ahler 4-manifold $(M,J,G,\Omega)$ and let $a_{H}$ be its Maslov 1-form. Then, \[
da_H=\Phi^{\ast}\bar\rho,
\]
where $\bar\rho$ is the Ricci form of $G$. 
\end{prop}
\begin{proof}
A straightforward computation gives:
\[
2\Phi^{\ast}\bar\rho(e_1,e_2)=\overline{\mbox{Ric}}(e_1,Je_2)=-G(\bar{R}(e_1,e_2)e_1,Je_1)+G(\bar{R}(e_1,e_2)e_2,Je_2)
\]
The Codazzi-Mainardi equations are:
\[
\bot \bar{R}(e_1,e_2)e_1=\overline\nabla_{e_1}h(e_1,e_2)-h(\nabla_{e_1}e_1,e_2)-h(\nabla_{e_1}e_2,e_1)-\overline\nabla_{e_2}h(e_1,e_1)+2h(\nabla_{e_2}e_1,e_1),
\]
\[
\bot \bar{R}(e_1,e_2)e_2=\overline\nabla_{e_1}h(e_2,e_2)-2h(\nabla_{e_1}e_2,e_2)+h(\nabla_{e_2}e_1,e_2)
-\overline\nabla_{e_2}h(e_1,e_2)+h(\nabla_{e_2}e_2,e_1),
\]
where $\overline\nabla$ and $\nabla$ denote the Levi-Civita connection with respect to $G$ and $\Phi^{\ast}G$, respectively.

Then,
\begin{eqnarray}
2\Phi^{\ast}\bar\rho(e_1,e_2)&=&-G(\bar{R}(e_1,e_2)e_1,Je_1)+G(\bar{R}(e_1,e_2)e_2,Je_2)\nonumber \\
&=& -G(\overline\nabla_{e_1}h(e_1,e_2),Je_1)+G(h(\nabla_{e_1}e_1,e_2),Je_1)
+G(h(\nabla_{e_1}e_2,e_1),Je_1)\nonumber \\
&&+G(\overline\nabla_{e_2}h(e_1,e_1),Je_1)-2G(h(\nabla_{e_2}e_1,e_1),Je_1)\nonumber \\
&&+G(\overline\nabla_{e_1}h(e_2,e_2),Je_2)-2G(h(\nabla_{e_1}e_2,e_2),Je_2)+G(h(\nabla_{e_2}e_1,e_2),Je_2)\nonumber \\
&&-G(\overline\nabla_{e_2}h(e_1,e_2),Je_2)+G(h(\nabla_{e_2}e_2,e_1),Je_2)\nonumber \\
&=&G(\overline\nabla_{e_2}h(e_1,e_1),Je_1)-G(\overline\nabla_{e_1}h(e_1,e_2),Je_1)+G(\overline\nabla_{e_1}h(e_2,e_2),Je_2) \nonumber \\
&&-G(\overline\nabla_{e_2}h(e_1,e_2),Je_2)+G(h(e_1,e_2),J\nabla_{e_1}e_1)+G(h(e_1,e_1),J\nabla_{e_1}e_2)\nonumber\\
&&\qquad -2G(h(e_1,e_1),J\nabla_{e_2}e_1)-2G(h(e_2,e_2),J\nabla_{e_1}e_2)\nonumber\\
&&\qquad\qquad +G(h(e_2,e_2),J\nabla_{e_2}e_1)+G(h(e_1,e_2),J\nabla_{e_2}e_2).\nonumber
\end{eqnarray}
Set 
\[
A:=G(\overline\nabla_{e_2}h(e_1,e_1),Je_1)-G(\overline\nabla_{e_1}h(e_1,e_2),Je_1)+G(\overline\nabla_{e_1}h(e_2,e_2),Je_2) 
-G(\overline\nabla_{e_2}h(e_1,e_2),Je_2),
\]
so that,
\begin{eqnarray}
2\Phi^{\ast}\bar\rho(e_1,e_2)&=&A+G(h(e_1,e_2),J\nabla_{e_1}e_1)+G(h(e_1,e_1),J[e_1,e_2])\nonumber\\
&&\qquad -G(h(e_1,e_1),J\nabla_{e_2}e_1)-G(h(e_2,e_2),J\nabla_{e_1}e_2)\nonumber\\
&&\qquad\qquad +G(h(e_2,e_2),J[e_2,e_1])+G(h(e_1,e_2),J\nabla_{e_2}e_2)\nonumber\\
&=&A+G(2H,J[e_1,e_2])+G(h(e_1,e_2),J\nabla_{e_1}e_1)\nonumber\\
&& -G(h(e_1,e_1),J\nabla_{e_2}e_1)-G(h(e_2,e_2),J\nabla_{e_1}e_2)+G(h(e_1,e_2),J\nabla_{e_2}e_2)\nonumber\\
&=&2A-2G(JH,[e_1,e_2])+\overline\nabla_{e_1}G(h(e_1,e_1),Je_2) -\overline\nabla_{e_2}G(h(e_1,e_1),Je_1)\nonumber\\
&&\qquad+\overline\nabla_{e_2}G(h(e_2,e_2),Je_1)-\overline\nabla_{e_1}G(h(e_2,e_2),Je_2).\nonumber
\end{eqnarray}

Note that,
\begin{eqnarray}
A&=&\overline\nabla_{e_2}G(h(e_1,e_1),Je_1)-G(h(e_1,e_1),J\nabla_{e_2}e_1)\nonumber \\
&&\quad-\overline\nabla_{e_1}G(h(e_1,e_1),Je_2)+G(h(e_1,e_2),J\nabla_{e_1}e_1)\nonumber\\
&&\qquad+\overline\nabla_{e_1}G(h(e_2,e_2),Je_2)-G(h(e_2,e_2),J\nabla_{e_1}e_2)\nonumber\\
&&\quad\qquad-\overline\nabla_{e_2}G(h(e_2,e_2),Je_1)+G(h(e_1,e_2),J\nabla_{e_2}e_2).\nonumber
\end{eqnarray}

We now have,
\begin{eqnarray}
2\Phi^{\ast}\bar\rho(e_1,e_2)&=&2A-2G(JH,[e_1,e_2])+\overline\nabla_{e_1}G(h(e_1,e_1),Je_2) -\overline\nabla_{e_2}G(h(e_1,e_1),Je_1)\nonumber\\
&&\qquad+\overline\nabla_{e_2}G(h(e_2,e_2),Je_1)-\overline\nabla_{e_1}G(h(e_2,e_2),Je_2)\nonumber\\
&=&-2G(JH,[e_1,e_2])+\overline\nabla_{e_2}G(h(e_1,e_1),Je_1)-\overline\nabla_{e_1}G(h(e_1,e_1),Je_2)\nonumber\\
&&+\overline\nabla_{e_1}G(h(e_2,e_2),Je_2)-\overline\nabla_{e_2}G(h(e_2,e_2),Je_1)-2G(h(e_1,e_1),J\nabla_{e_2}e_1)
\nonumber\\
&&+2G(h(e_1,e_2),J\nabla_{e_1}e_1)-2G(h(e_2,e_2),J\nabla_{e_1}e_2)+2G(h(e_1,e_2),J\nabla_{e_2}e_2)\nonumber\\
&=&\overline\nabla_{e_2}G(2H,Je_1)-\overline\nabla_{e_1}G(2H,Je_2)-2G(JH,[e_1,e_2]) -2G(h(e_1,e_1),J\nabla_{e_2}e_1)\nonumber\\
&&+2G(h(e_1,e_2),J\nabla_{e_1}e_1)-2G(h(e_2,e_2),J\nabla_{e_1}e_2)+2G(h(e_1,e_2),J\nabla_{e_2}e_2)\nonumber\\
&=&2\overline\nabla_{e_1}G(JH,e_2)-2\overline\nabla_{e_2}G(JH,e_1)-2G(JH,[e_1,e_2]) -2G(h(e_1,e_1),J\nabla_{e_2}e_1)\nonumber\\
&&+2G(h(e_1,e_2),J\nabla_{e_1}e_1)-2G(h(e_2,e_2),J\nabla_{e_1}e_2)+2G(h(e_1,e_2),J\nabla_{e_2}e_2)\nonumber\\
&=&2da_H(e_1,e_2) -2G(h(e_1,e_1),J\nabla_{e_2}e_1)+2G(h(e_1,e_2),J\nabla_{e_1}e_1)\nonumber\\
&&\qquad-2G(h(e_2,e_2),J\nabla_{e_1}e_2)+2G(h(e_1,e_2),J\nabla_{e_2}e_2)\label{e:semjhu}
\end{eqnarray}

Let $\omega^k_{ij}$ be the functions defined by 
\[
\nabla_{e_i}e_j=\sum_{k=1}^2\omega^k_{ij}e_k.
\]
Since $|e_1|^2=-|e_2|^2=1$, we have 
\[
\omega^2_{11}=\omega^1_{12},\qquad \omega^2_{21}=\omega^1_{22},
\]
from which we deduce that
\begin{equation}\label{e:semjhu1}
G(h(e_1,e_1),J\nabla_{e_2}e_1)=G(h(e_1,e_2),J\nabla_{e_2}e_2),
\end{equation}
and
\begin{equation}\label{e:semjhu2}
G(h(e_1,e_2),J\nabla_{e_1}e_1)=G(h(e_2,e_2),J\nabla_{e_1}e_2).
\end{equation}
The proposition follows by substituting (\ref{e:semjhu1}) and (\ref{e:semjhu2}) into (\ref{e:semjhu}).

\end{proof}

\bigskip\bigskip

 Department of Mathematics and Statistics, University of Cyprus, P.O. Box 20537, 1678 Nicosia, Cyprus 

E-mail: \mail{georgiou.g.nicos}{ucy.ac.cy}

\end{document}